\documentclass[11pt]{amsart}

\usepackage[english]{babel}
\usepackage[utf8x]{inputenc}
\usepackage[T1]{fontenc}
\usepackage{arydshln}

\usepackage{amsmath,amsthm,amsfonts,amssymb,comment,enumerate,wasysym}
\usepackage{graphicx}
\usepackage{tikz}
\usetikzlibrary{snakes,fit}
\usepackage[colorinlistoftodos]{todonotes}
\usepackage[colorlinks=true, allcolors=blue]{hyperref}
\usepackage{blkarray}

\newtheorem*{maintheorem}{Main Theorem}
\newtheorem{theorem}{Theorem}[section]
\newtheorem{lemma}[theorem]{Lemma}
\newtheorem{corollary}[theorem]{Corollary}	
\newtheorem{proposition}[theorem]{Proposition}
\newtheorem{claim}[theorem]{Claim}

\theoremstyle{definition}
\newtheorem{example}[theorem]{Example}
\newtheorem{definition}[theorem]{Definition}
\newtheorem{remark}[theorem]{Remark}

\newcommand{\lf}{\lfloor}
\newcommand{\rf}{\rfloor}
\newcommand{\Lf}{\left\lfloor}
\newcommand{\Rf}{\right\rfloor}
\newcommand{\Floor}[1]{\Lf{#1}\Rf}
\newcommand{\thirdfloor}[1]{\Floor{\frac{#1}{3}}}

\newcommand{\config}{\textsc{Conf}}
\newcommand{\Lk}{\text{Lk}}
\newcommand{\J}{\mathcal{J}}
\newcommand\addvmargin[1]{
  \node[fit=(current bounding box),inner ysep=#1,inner xsep=0]{};
}

\begin{document}

\title[Connectivity at infinity for the braid group of a $K_{m,n}$]{Connectivity at infinity for the braid group of a complete bipartite graph}

\author[Mazur]{Kristen~Mazur}
\email{kmazur@elon.edu}
\address{Department of Mathematics,
Elon University,
Elon NC 27244}

\author[McCammond]{Jon~McCammond} 
\email{jon.mccammond@math.ucsb.edu}
\address{Dept. of Mathematics,
  University of California, Santa Barbara, CA 93106}

\author[Meier]{John~Meier}
\email{meierj@lafayette.edu}
\address{Department of Mathematics,
Lafayette College,
Easton PA 18042}

\author[Rohatgi]{Ranjan~Rohatgi}
\email{rrohatgi@saintmarys.edu}
\address{Department of Mathematics and Computer Science,
Saint Mary's College,
Notre Dame IN 46556}

\date{\today}

\begin{abstract}
The graph braid group of a complete bipartite graph is the fundamental group of a configuration space of points on the graph, which is a CAT(0) cube complex.  We combine an analysis of the topology of links of vertices in this complex, the description of a hidden symmetry among the parameters, and known results from the literature to explicitly compute the exact degree to which these complexes and groups are connected at infinity.
\end{abstract}

\subjclass[2010]{20F65, 57M07}

\maketitle

\section{Introduction}

The classical braid group $B_r$ is the fundamental group of the configuration space of $r$ disjoint indistinguishable points in a disk.  There is a similar combinatorial construction that begins with a fixed natural number $r$ and a finite, simple graph $\Gamma$ and then produces an associated cell complex corresponding to $r$ points moving on $\Gamma$ where the points stay reasonably far apart. (See for example \cite{AbramsGhrist} and \cite{FarleySab}.) Following Abrams and Ghrist, we refer to the $r$ points on the graph $\Gamma$ as \emph{robots}.   The fundamental group of this cell complex is called a graph braid group, and in this article we describe the connectivity at infinity of graph braid groups when the graph $\Gamma$ is a complete bipartite graph.  The case of a complete graph is settled in~\cite{MZ}. 

In order to determine connectivity at infinity for the graph braid group we describe the connectivity at infinity of the universal cover of the configuration space mentioned above. Recall that a cell complex $K$ is \emph{$n$-connected} if its homotopy groups $\pi_i(K)$ are trivial for $0 \leq i \leq n$.  Thus being $0$-connected is equivalent to being path connected, and being $1$-connected is often referred to as being \emph{simply connected}.  Roughly speaking, a contractible complex is said to be \emph{$n$-connected at infinity} if the complements of finite subcomplexes are $n$-connected. (A rigorous definition is given in Section~\ref{sec:configurations}.)  The property of being $0$-connected at infinity is often referred to as being one-ended, as there is essentially one connected component in the complement of finite subcomplexes.  The property of being $1$-connected at infinity is often described as being \emph{simply connected at infinity}. If $K$ is a finite, aspherical complex then the connectivity at infinity of the universal cover $\widetilde{K}$ is a quasi-isometry invariant of $K$, and it is directly related to the group cohomology of $\pi_1(K)$ and duality properties of groups. 

We begin by stating our main result.

\begin{maintheorem}
Consider $r$ robots on a complete bipartite graph $K_{n,N}$. Let $R$ be the number of open vertices on $K_{n,N}$ so that $r+R=n+N$ is the total number of vertices.  Moreover, to avoid trivial cases, assume that $r$, $n$, $N$ and $R$ are all at least $2$.  Let $\ell_0 = \min\{r,R,n,N\}$, $\ell_1=\min\{r,R\}+\min\{n,N\}+1$, and $\ell_2=r+R=n+N$.  If 
\[\ell = \min\left\{
\ell_0, \thirdfloor{\ell_1}, \thirdfloor{\ell_2}
\right\},\]
then the universal cover of the combinatorial configuration space of $r$ robots on $K_{n,N}$ is $(\ell - 2)$-connected at infinity but not $(\ell - 1)$-connected at infinity.  
\end{maintheorem}

For example, if we have two robots moving on a $K_{n,N}$ with $n$ and $N$ each greater or equal to 3,  then $\ell = 2$ and we see that the universal cover of the configuration space is one-ended. 

\begin{remark} The universal cover of a configuration space can be thought of as the space of all possible robot motions on the graph, and so determining the connectivity at infinity is essentially a result about robot motion planning when certain motions are blocked for some finite period of time.
\end{remark}

In order to prove our Main Theorem, we use a known result that relates the connectivity at infinity to the connectivity of vertex links.  The following theorem, which is a combination of results proved in \cite{BradyMeier} and \cite{converse}, states this close connection.

\begin{theorem}[Vertex links and infinity]\label{thm:link-infinity}
Let $X$ be a finite, locally \emph{CAT(0)} cube-complex with universal cover $\widetilde{X}$, and fix an integer $n \ge -1$. Assume the link of every vertex in $X$ is $n$-connected, and remains $n$-connected when any closed simplex is removed. Further, assume that the link of some vertex $v$ has $H_{n+1}(\Lk(v)) \neq 0$. Then $\widetilde{X}$ is $n$-connected at infinity but not $(n+1)$-connected at infinity.
\end{theorem}

In \cite{birthday}, Abrams and Ghrist establish the fact that a combinatorial  configuration space for $r$ robots on any finite simple graph is a finite, locally $CAT(0)$ cube-complex. Hence, in order to prove the Main Theorem we need to determine the maximum connectivity shared by every vertex link in the configuration space of $r$ robots on $K_{n,N}$, which is determined by the minimum non-connectivity of some vertex link. We show that these links  are chessboard complexes or joins of chessboard complexes, which are well studied combinatorial objects.  Moreover, we establish some surprising isomorphisms between finite covers of these configuration spaces that allow us to reduce our analysis to the case where $r\leq n\leq N\leq R$, and then successfully compute the minimum non-connectivity in this context.

The paper is organized as follows. Section~\ref{sec:configurations} describes the configuration spaces and connectivity at infinity. Section~\ref{sec:symmetries} discusses  hidden symmetries that give isomorphisms that allow us to reduce to the case where the parameters are ordered as follows: $r \leq n\leq N\leq R$. We then describe the structure of the vertex links and show that they are joins of chessboard complexes in Section~\ref{sec:vertex-links}, and finally we prove the Main Theorem in Section~\ref{sec:computations}.

\section{Configuration Spaces}\label{sec:configurations}

Throughout this paper we consider $r$ robots moving on a complete bipartite graph $K_{n,N}$.  In order for it to be possible for the robots to simultaneously occupy distinct vertices of $K_{n,N}$, it must be the case that $r \leq n + N$ and we assume this condition on the parameters always holds.  

We think of the complete bipartite graph $K_{n,N}$ as a simplicial complex whose non-empty simplices $\sigma$ are its vertices and edges.  If we have $r$ robots moving on $K_{n,N}$ then each robot is on a vertex or an edge.  In order to prohibit collisions, we assume that not only are no two robots allowed on any one vertex or edge, but also that if a robot is on an edge then both bounding vertices are robot-free.  Since we will not distinguish between different locations on a single edge, every position of $r$ robots is then determined by a collection of $r$ disjoint vertices and edges. 

\begin{definition}[Discrete configuration spaces]\label{def:config-spaces}
For each complete bipartite graph $K = K_{n,N}$ we define the \emph{discrete configuration space} $\config_r(n,N)$ to be the  subcomplex of \[\underbrace{K \times \ldots \times K}_{r \text{ copies}}\] consisting of products of vertices and edges $\sigma_1 \times \cdots \times \sigma_r$ where for each pair of distinct, closed simplices $\sigma_i \cap \sigma_j = \emptyset$.
\end{definition}

\begin{remark}[Vertex versus $0$-cell]\label{rem:vertex-vs-0-cell}
Because we build configuration spaces based on graphs, the terms ``vertex'' and ``edge'' are potentially ambiguous.   Do they refer to a piece of the graph or a piece of the configuration space? In this article, we describe the graph $K_{n,N}$ as having ``vertices'' and ``edges'' while the configuration spaces have ``$i$-cells'' instead.  We make this distinction so that it is always clear to the reader which of the two types of cell complexes we are referring to.  The only exception we make to this convention is the standard terminology ``vertex link'' (see Remark~\ref{rem:vertex-link}).
\end{remark}

\begin{remark}[Really small parameters]\label{rem:AtLeast2}
The cases where $\min\{r, n, N\} = 1$ are elementary to resolve, and we remove them from the statements of our main results. When $r=1$, the configuration space is the same as the underlying graph 
\[
\config_1(n,N) \simeq K_{n,N}
\]
and the graph braid group is its fundamental group, a free group. The symmetries described in Section~\ref{sec:symmetries} show that these graph braid groups are also free groups when $n=1$ or $N=1$.  For the remainder of the article we assume that $\min\{r,n,N\} \geq 2$.
\end{remark}

\begin{example}[Small parameters]\label{ex:small-parameters}
One of the classic examples is the configuration space $\config_2(3,3)$ for $2$ robots on a $K_{3,3}$. Because it is a subset of $K_{3,3} \times K_{3,3}$, $\config_2(3,3)$ is a $2$-dimensional complex.  The $0$-cells of $\config_2(3,3)$ correspond to having both robots on vertices of $K_{3,3}$; $1$-cells correspond to one robot sitting on a vertex while the other moves along an edge between two vertices; and the $2$-cells, which are squares, correspond to both robots moving along two distinct edges with four distinct vertices.  The neighborhood of a $0$-cell $v \in \config_2(3,3)$ is determined by whether the associated configuration has two robots  on vertices from a single side of $K_{3,3}$ or on different sides.  If both robots are on the same side, then there are $6$ squares joined in a cycle about $v$.  If the robots are on opposite sides then there are $4$ squares joined in a cycle about $v$.  Once this local structure has been determined it is not too difficult to establish that $\config_2(3,3)$ is a closed hyperbolic surface.  For the details, see  \cite{birthday}. We note that it is  uncommon to have configuration space $\config_r(n,N)$ be a manifold.  See Example~\ref{ex:NotMfd}. 
\end{example}

In general, since the configuration space $\config_r(n,N)$ is a subcomplex of a product of graphs, it has a natural cell structure where each cell is a product of vertices and edges.

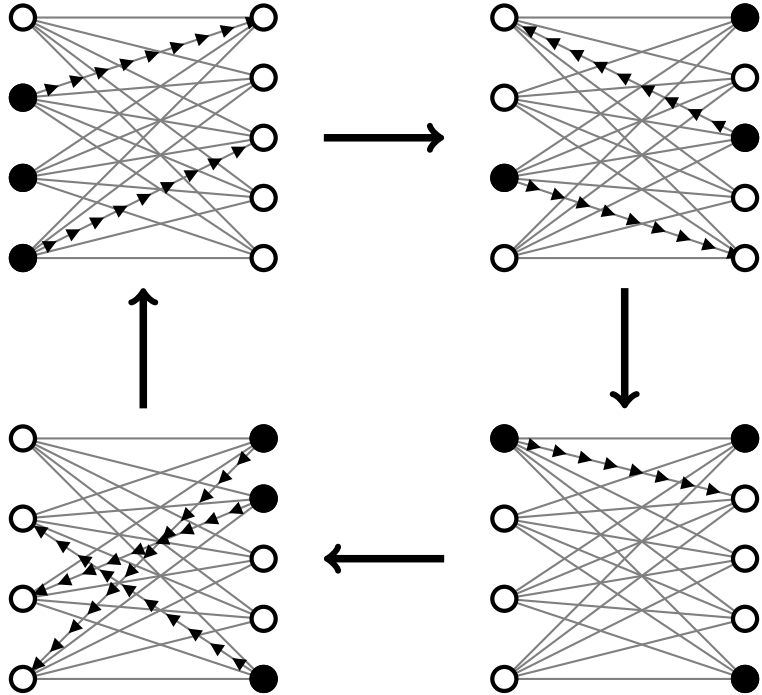
\begin{figure}[hptb]
\begin{tikzpicture} [scale=0.8, thick]
\usetikzlibrary{snakes}
	
\draw [->, line width=.1cm] (-1,2)--(1,2);
\draw [<-, line width=.1cm] (-1,-5)--(1,-5);
\draw [->, line width=.1cm] (4,-.5)--(4,-2.5);
\draw [<-, line width=.1cm] (-4,-.5)--(-4,-2.5);

\begin{scope} [xshift=-4cm]
\foreach \x in {0,4/3,8/3,4} \foreach \y in {0,1,2,3,4} 
\draw [gray] (-2,\x)--(2,\y);
\end{scope}

\begin{scope} [xshift=4cm]
\foreach \x in {0,4/3,8/3,4} \foreach \y in {0,1,2,3,4} 
\draw [gray] (-2,\x)--(2,\y);
\end{scope}

\begin{scope} [xshift=-4cm,yshift=-7cm]
\foreach \x in {0,4/3,8/3,4} \foreach \y in {0,1,2,3,4} 
\draw [gray] (-2,\x)--(2,\y);
\end{scope}

\begin{scope} [xshift=4cm, yshift=-7cm]
\foreach \x in {0,4/3,8/3,4} \foreach \y in {0,1,2,3,4} 
\draw [gray] (-2,\x)--(2,\y);
\end{scope}

\draw [ultra thick, snake=triangles] (-6,0)--(-2,2);
\draw [ultra thick, snake=triangles] (-6,8/3)--(-2,4);

\begin{scope} [xshift=8cm]
\draw [ultra thick, snake=triangles] (-2,2)--(-6,4);
\draw [ultra thick, snake=triangles] (-6,4/3)--(-2,0);
\end{scope}

\begin{scope} [xshift=8cm,yshift=-7cm]
\draw [ultra thick, snake=triangles] (-6,4)--(-2,3);
\end{scope}

\begin{scope} [yshift=-7cm]
\draw [ultra thick, snake=triangles] (-2,4)--(-6,0);
\draw [ultra thick, snake=triangles] (-2,3)--(-6,4/3);
\draw [ultra thick, snake=triangles] (-2,0)--(-6,8/3);
\end{scope}

\begin{scope} [xshift=-4cm]
\foreach \x in {0,4/3,8/3,4} \draw [ultra thick, fill=white] (-2,\x) circle [radius=0.2];
\foreach \x in {0,1,2,3,4} \draw [ultra thick, fill=white] (2,\x) circle [radius=0.2];
\end{scope}

\begin{scope} [xshift=4cm]
\foreach \x in {0,4/3,8/3,4} \draw [ultra thick, fill=white] (-2,\x) circle [radius=0.2];
\foreach \x in {0,1,2,3,4} \draw [ultra thick, fill=white] (2,\x) circle [radius=0.2];
\end{scope}

\begin{scope} [xshift=-4cm,yshift=-7cm]
\foreach \x in {0,4/3,8/3,4} \draw [ultra thick, fill=white] (-2,\x) circle [radius=0.2];
\foreach \x in {0,1,2,3,4} \draw [ultra thick, fill=white] (2,\x) circle [radius=0.2];
\end{scope}

\begin{scope} [xshift=4cm, yshift=-7cm]
\foreach \x in {0,4/3,8/3,4} \draw [ultra thick, fill=white] (-2,\x) circle [radius=0.2];
\foreach \x in {0,1,2,3,4} \draw [ultra thick, fill=white] (2,\x) circle [radius=0.2];
\end{scope}

\draw [ultra thick, fill=black] (-6,0) circle [radius=0.2];
\draw [ultra thick, fill=black] (-6,4/3) circle [radius=0.2];
\draw [ultra thick, fill=black] (-6,8/3) circle [radius=0.2];

\begin{scope} [xshift=8cm]
\draw [ultra thick, fill=black] (-2,2) circle [radius=0.2];
\draw [ultra thick, fill=black] (-6,4/3) circle [radius=0.2];
\draw [ultra thick, fill=black] (-2,4) circle [radius=0.2];
\end{scope}

\begin{scope} [xshift=8cm,yshift=-7cm]
\draw [ultra thick, fill=black] (-6,4) circle [radius=0.2];
\draw [ultra thick, fill=black] (-2,0) circle [radius=0.2];
\draw [ultra thick, fill=black] (-2,4) circle [radius=0.2];
\end{scope}

\begin{scope} [yshift=-7cm]
\draw [ultra thick, fill=black] (-2,3) circle [radius=0.2];
\draw [ultra thick, fill=black] (-2,0) circle [radius=0.2];
\draw [ultra thick, fill=black] (-2,4) circle [radius=0.2];
\end{scope}
\end{tikzpicture}
\caption{An example of three robots moving on a $K_{4,5}$.\label{fig:RobotMotion}}
\end{figure}

\begin{example}[A loop]
The robot motion depicted in Figure~\ref{fig:RobotMotion} shows a loop in the cubical complex $\config_3(4,5)$ associated to three robots on the complete bipartite graph $K_{4,5}$.  The four configurations depicted, with the three black dots representing the locations of the robots, correspond to $0$-cells in $\config_3(4,5)$.  The edges with arrows indicate how some subset of the robots are about to move as we transition to the next figure.

If we start at the $0$-cell corresponding to the configuration shown on the lower right, one robot is moving and this means this portion of the loop traverses an edge in the configuration space to the $0$-cell corresponding to the configuration shown on the lower left.  From here, three robots are moving. This means that this portion of the loop traverses a diagonal of a $3$-cube in the configuration space to the $0$-cell corresponding to the configuration shown on the upper left.  The next robot motion corresponds to a portion of the loop which traverses the diagonal of a square (since exactly two robots are moving), as does the final motion.
\end{example}

Because the cells are products of vertices and edges, the configuration space $\config_r(n,N)$ is a cubical cell complex, and assigning the metric of unit Euclidean cubes to each cell gives $\config_r(n,N)$ a piecewise Euclidean metric. A well-known result of Gromov can be applied to prove that the universal cover of $\config_r(n,N)$ is non-positively curved. (See \cite{BridsonHaefliger} for background information about CAT(0) complexes and Gromov's Theorem.) See also Theorem~3.3 and Corollary~3.4 in \cite{birthday}.

\begin{theorem}[CAT(0) cube complex]\label{thm:cat-cube-cplx}
The space $\config_r(n,N)$ is a locally \emph{CAT(0)} cube complex, which makes $X=\config_r(n,N)$ a classifying space, i.e. it is a $K(G,1)$ where $G= \pi_1(X)$ is its fundamental group.
\end{theorem}

Knowing that $\config_r(n,N)$ is a $K(G,1)$ has a number of consequences, as topological features of classifying spaces are invariants of the fundamental group. In particular, the connectivity at infinity of the universal cover of $\config_r(n,N)$ is an invariant of the fundamental group of $\config_r(n,N)$.  (See \cite{brown} and \cite{geog} for results along these lines.) 

We conclude this section by recalling the definition of connectivity at infinity. Let $\widetilde{K}$ be the universal cover of a finite, aspherical cell complex $K$. The topology at infinity of $\widetilde{K}$ consists of those topological features that are persistently present in the complements of finite subcomplexes of $\widetilde{K}$.  In particular, we have the following definition.

\begin{definition}[Connectivity at infinity]\label{def:conn-infinity}
The universal cover $\widetilde{K}$ is \emph{$n$-connected at infinity} if given any compact $C \subset \widetilde{K}$ there is a compact $D \subset \widetilde{K}$ such that any map $\phi: S^i \rightarrow \widetilde{K}\setminus D$ extends to a map $\hat{\phi}: B^{i+1} \rightarrow \widetilde{K}\setminus C$ for all $-1 \leq i \leq n$.  
\end{definition}

Euclidean space, $\mathbb{R}^n$, is $(n-2)$-connected at infinity but not $(n-1)$-connected at infinity. We note that the usual definition of being $(-1)$-connected is that the space is non-empty, which implies that being $(-1)$-connected at infinity is equivalent to the statement that $\widetilde{K}$ is infinite.  As we remarked in the introduction, the property of being $0$-connected at infinity is commonly referred to as having \emph{one-end}, and being $1$-connected at infinity is usually termed \emph{simply connected at infinity}.

\begin{remark}[Other topological invariants] Other topological invariants of graph braid groups have been studied. The work in \cite{KoPark} highlights just how complicated it is to determine even the first homology group for graph braid groups in terms of the structure of the underlying graph. Almost all results about homology and cohomology occur only in specific contexts, for example, when the underlying graph is a tree (see for example \cite{FarleySab}, \cite{FarleyHomology}, \cite{FarleyCohoTree}, \cite{Scheirer}, and \cite{Ramos}).
\end{remark}

\section{Hidden symmetries}\label{sec:symmetries}

In this section we describe three symmetries among the parameters used to define $\config_r(n,N)$.  The first two are relatively obvious, while the third is somewhat more surprising.  The main purpose of introducing these three symmetries is to be able to reduce to the case where the parameters satisfy the inequalities $r \leq n \leq N$.  In other words we can restrict our attention to those cases where all robots fit on the smaller side of the bipartite graph.  These inequalities are crucial to the explicit computations carried out in Section~\ref{sec:computations}.

The first symmetry is a result of the obvious left-right symmetry of $K_{n,N}$.

\begin{lemma}[Left and right]\label{lem:swap-left-right}
The configuration spaces $\config_r(n,N)$ and $\config_r(N,n)$ are
isomorphic.
\end{lemma}

The second symmetry follows from the introduction of ghosts whose motions mirror those of the robots.

\begin{definition}[Ghosts] 
In the configuration space of $r$ robots on a $K_{n,N}$, we view every unoccupied vertex as being occupied by a \emph{ghost}.  If we let $R$ denote the number of ghosts, then $r+R = n+N =T$ where $T$ is the total number of vertices of $K_{n,N}$.  In addition, when a robot moves from one vertex to an ``unoccupied'' vertex we view the ghost at the unoccupied vertex as moving in the opposite direction so that it exchanges places with the robot.
\end{definition}

Once ghosts are introduced so that they move in conjunction with and opposite to the motion of the robots, it becomes clear that the robot configuration space corresponds in a natural way to another configuration space where the old ghosts are viewed as the new robots and the old robots are viewed as the new ghosts.  This proves the following lemma.

\begin{lemma}[Robots and ghosts]\label{lem:swap-robot-ghost}
Let $n + N = T$ and let $R = T-r$.  Then $\config_r(n,N)$ is isomorphic to $\config_R(n,N)$.  
\end{lemma}

The third symmetry is both more complicated to state accurately and a bit more difficult to prove. The idea can best be illustrated by coarsely viewing robot configurations as a non-negative integer solution to a 2 by 2 row-column sum problem.

\begin{definition}[Solutions]\label{def:solutions}
Let $r,R,n,N$ be positive integers with $r+R = n+N$ and let $(a,b,c,d)$ be a $4$-tuple of non-negative integers.  When $a+b=r$, $c+d=R$, $a+c=n$ and $b+d=N$ we say that $(a,b,c,d)$ is a \emph{solution to the $(r,R,n,N)$ row and column sum problem} or, more simply, that $(a,b,c,d)$ is an \emph{$(r,R,n,N)$-solution}. The terminology comes from the fact that the $2 \times 2$ matrix with entries $a$, $b$, $c$ and $d$ has row sums equal to $r$ and $R$ and column sums equal to $n$ and $N$. The common value $T=r+R=n+N=a+b+c+d$ is the total of all four entries. See the array below.
\[\begin{blockarray}{ccc}
\begin{block}{[cc]c}
 a  & b & r \\
 c & d &  R \\
\end{block}
n & N &   T\\
\end{blockarray}\]  
\end{definition}

\begin{remark}[Robot configurations as solutions]
Notice that every configuration of $r$ robots on the vertices of a $K_{n,N}$ yields an $(r,R,n,N)$-solution. Let $a$ be the number of robots on the $n$-side of $K_{n,N}$, let $b$ be the number of robots on the $N$-side, and let $c$ and $d$ be the number of ghosts on each side. Then the $4$-tuple $(a,b,c,d)$ is an $(r,R,n,N)$-solution. (See Figure~\ref{fig:abcd}.) 
\end{remark}

\begin{figure}
\begin{tikzpicture} [scale=0.5, thick]
\foreach \x in {2,3,4} \draw [ultra thick, fill=black] (-2,\x) circle [radius=0.22];
\foreach \x in {2,3,4} \draw [ultra thick, fill=black] (2 -.2,\x -.2) rectangle (2+.2, \x +.2);
\foreach \x in {-1,0,1} \draw [fill=gray, gray] (-2,\x) circle [radius=.05];
\foreach \x in {-1,0,1} \draw [fill=gray, gray] (2,\x) circle [radius=.05];
\foreach \x in {-2,-3,-4} \draw [ultra thick, fill=white] (-2,\x) circle [radius=0.22];
\foreach \x in {-2,-3,-4} \draw [ultra thick, fill=white] (2 -.2,\x -.2) rectangle (2+.2, \x +.2);
\draw [snake=brace, ultra thick] (-3,.75)--++(0,3.5);
\draw [snake=brace, ultra thick] (-3,-4.25)--++(0,4.5);
\draw [snake=brace, mirror snake, ultra thick] (3,-.25)--++(0,4.5);
\draw [snake=brace, mirror snake, ultra thick] (3,-4.25)--++(0,3.5);
\node at (-3.6,2.5) {\Large $a$};
\node at (3.6,2) {\Large $b$};
\node at (-3.6,-2) {\Large $c$};
\node at (3.6,-2.5) {\Large $d$};
\node at (-2,-5.5) {\Large $n$ spots};
\node at (2,-5.5) {\Large $N$ spots};
\end{tikzpicture}
\caption{The robots are positioned at the black spots and the ghosts are at the white spots. The circles and squares denotes  the two parts of the bipartite graph.\label{fig:abcd}}
\end{figure}
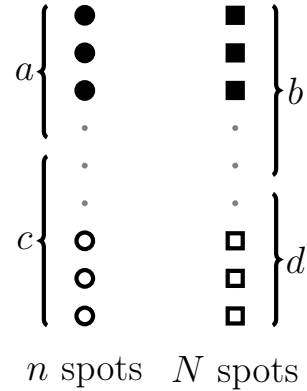

\begin{definition}[Configuration type]\label{def:config-type}
The \emph{type} of a configuration of $r$ robots on the vertices of $K_{n,N}$ is the associated $(r,R,n,N)$-solution. If $(a,b,c,d)$ is the $(r,R,n,N)$-solution of a configuration of $r$ robots on the vertices of $K_{n,N}$ then we say that such a configuration is \textit{type} $(a,b,c,d)$. 
\end{definition} 

\begin{remark}[Symmetries reinterpreted]
The eight symmetries of a square can be applied to a row and column sum problem and to its set of solutions. Notice that the symmetry listed in Lemma~\ref{lem:swap-left-right} corresponds to exchanging the two columns.  In other words, we can identify solutions of 
\[
\begin{blockarray}{ccc}
\begin{block}{[cc]c}
 \cdot  & \cdot & r \\
 \cdot & \cdot &  R \\
\end{block}
n & N &   T\\
\end{blockarray}
\hspace{.5in} \text{ with solutions of } \hspace{.5in}
\begin{blockarray}{ccc}
\begin{block}{[cc]c}
 \cdot  & \cdot & r \\
 \cdot & \cdot &  R \\
\end{block}
N & n &   T\\
\end{blockarray}
\]

Similarly, the symmetry listed in Lemma~\ref{lem:swap-robot-ghost} corresponds to exchanging the two rows.  In other words, we can identify 
solutions of 
\[
\begin{blockarray}{ccc}
\begin{block}{[cc]c}
 \cdot  & \cdot & r \\
 \cdot & \cdot &  R \\
\end{block}
n & N &   T\\
\end{blockarray}
\hspace{.5in} \text{ with solutions of } \hspace{.5in}
\begin{blockarray}{ccc}
\begin{block}{[cc]c}
 \cdot  & \cdot & R \\
 \cdot & \cdot &  r \\
\end{block}
n & N &   T\\
\end{blockarray}
\]
\end{remark}

\begin{remark}[The third symmetry]
The third symmetry that we introduce can be interpreted as transposing solutions; it identifies solutions of
\[
\begin{blockarray}{ccc}
\begin{block}{[cc]c}
 \cdot  & \cdot & r \\
 \cdot & \cdot &  R \\
\end{block}
n & N &   T\\
\end{blockarray}
\hspace{.5in} \text{ with solutions of } \hspace{.5in}
\begin{blockarray}{ccc}
\begin{block}{[cc]c}
 \cdot  & \cdot & n \\
 \cdot & \cdot &  N \\
\end{block}
r & R &   T\\
\end{blockarray}
\]
This symmetry  leads to a local isomorphism between a neighborhood of a $0$-cell in  $\config_r(n,N)$ and a neighborhood of some $0$-cell in  $\config_n(r,R)$.  This is somewhat surprising since these two configuration spaces are not identical.  They do not even have the same number of $0$-cells (see Example~\ref{ex:common-universal-cover}).
\end{remark}

In the remainder of this section we introduce additional features that help us to prove that while not identical, $\config_r(n,N)$ and  $\config_n(r,R)$ have a common, finite cover.  This implies the local isomorphism alluded to above and ultimately allows us to assume $r \leq n \leq N$. We develop a common covering space of $\config_r(n,N)$ and $\config_n(r,R)$ by making both robots and ghosts distinguishable rather than indistinguishable.  To formalize this we introduce a more precise label for each $0$-cell in the configuration space.  The $0$-cells in the version with distinguishable robots and distinguishable ghosts can be placed in bijection with $T \times T$ permutation matrices where $T = r+R = n+N$ is the total number of vertices in the bipartite graph $K_{n,N}$.  We note that the vertices of $K_{n,N}$ are distinguished, one from another, and the usual convention for this distinction comes from vertically ordering  the  vertices on the left and right.  Consider the following example.

\begin{figure}[hptb]
\begin{tikzpicture} [scale=1, thick]

\begin{scope} [xshift=-3.5cm]
\foreach \x in {1,2,3,4} \foreach \y in {0.5,1.5,2.5,3.5,4.5} 
\draw [very thick, gray] (-2,\x)--(2,\y);

\foreach \x in {1,2,3,4} \draw [ultra thick, fill=white] (-2,\x) circle [radius=0.4];
\foreach \x in {0.5,1.5,2.5,3.5,4.5} \draw [ultra thick, fill=white] (2,\x) circle [radius=0.4];

\node at (-2,4) {$p_1$}; \node at (-2.75,4) {$q_1$};
\node at (-2,3) {$P_2$};\node at (-2.75,3) {$q_2$};
\node at (-2,2) {$p_3$};\node at (-2.75,2) {$q_3$};
\node at (-2,1) {$P_4$};\node at (-2.75,1) {$q_4$};
\node at (2,4.5) {$P_6$}; \node at (2.8,4.5) {$Q_1$};
\node at (2,3.5) {$p_2$}; \node at (2.8,3.5) {$Q_2$};
\node at (2,2.5) {$P_3$};\node at (2.8,2.5) {$Q_3$};
\node at (2,1.5) {$P_1$};\node at (2.8,1.5) {$Q_4$};
\node at (2,0.5) {$P_5$};\node at (2.8,0.5) {$Q_5$};
\end{scope}

\end{tikzpicture}
\caption{A fully labelled configuration of three robots and six ghosts moving on a $K_{4,5}$.}
\label{fig:labelEverythingA}
\end{figure}

\begin{example}[A fully labelled configuration]\label{ex:fullyLabelSymmetry}
Consider the fully labelled configuration shown in Figure~\ref{fig:labelEverythingA} and the corresponding permutation matrix shown in Figure~\ref{fig:fourMatrices}.  The labels next to the vertices are the names of the vertices ($q_i$ or $Q_i$).  The label inside the circle is the name of the robot or ghost that occupies this vertex.  If we view $p_1$, $p_2$ and $p_3$ as robot names, this is a configuration of $3$ robots on a $K_{4,5}$. If we remove the subscripts on the $p_j$'s and the $P_j$'s, to make the robots and the ghosts indistinguishable, this becomes a vertex in the configuration space $\config_3(4,5)$.  The fully labelled configuration is encoded in the $9\times 9$ matrix in the upper lefthand corner of Figure~\ref{fig:fourMatrices}.  Once the subscripts on the $p_j$'s and $P_j$'s are removed, this configuration is encoded in the $2\times 9$ matrix in the lower lefthand corner. The top row indicates the locations of the robots and the bottom row indicates the locations of the ghosts.
\end{example}

\begin{figure}
\[
\begin{array}{r@{}c@{}c@{}c}
& q\hspace*{5.8em} Q\hspace*{.7em} & &  \\ [5pt]
  \begin{array}{c} 
 \\ 
 p\\
 \\ 
 \\
 \\
 P\\
 \\
 \\
 \\
 \end{array}
 &
\left[ \begin{array}{cccc;{2pt/2pt}ccccc}
                     1 & 0 & 0 & 0 & 0 & 0 & 0 & 0 & 0  \\
                     0 & 0 & 0 & 0 & 0 & 1 & 0 & 0 & 0  \\
                      0 & 0 & 1 & 0 & 0 & 0 & 0 &0 & 0  \\
                      \hdashline[2pt/2pt]
                      0 & 0 & 0 & 0 & 0 & 0 & 0 &1 & 0  \\
                     0 & 1 & 0 & 0 & 0 & 0 & 0 &0 & 0  \\
                      0 & 0 & 0 & 0 & 0 & 0 & 1 & 0 & 0  \\
                      0 & 0 & 0 & 1 & 0 & 0 & 0 &0 & 0  \\
                     0 & 0 & 0 & 0 & 0 & 0 & 0 &0 & 1  \\
                      0 & 0 & 0 & 0 & 1 & 0 & 0 &0 & 0  \\
                                  \end{array}\right] 
 &\ & 
  \left[\begin{array}{cc}
                       1 & 0 \\ 
                       0 & 1 \\
                       1 & 0 \\ \hdashline[2pt/2pt]
                        0 & 1 \\ 
                       1 & 0 \\
                       0 & 1 \\ 
                       1 & 0 \\ 
                       0 & 1 \\
                       0 & 1 \\
                      \end{array}\right] \\ 
 \ \\
                                  & 
\left[\begin{array}{cccc;{2pt/2pt}ccccc}
         1 & 0 & 1 & 0 & 0 & 1 & 0 & 0 & 0 \\
         0 & 1 & 0 & 1 & 1 & 0 & 1 & 1 & 1 \\
  \end{array}\right]
 & &\left[ \begin{array}{cc}
2 & 1\\ 
2 & 4\\
\end{array}\right]  \\ \\ 
  
\end{array}
\]    

\caption{The four matrices associated with the configuration shown in Figure~\ref{fig:labelEverythingA}.\label{fig:fourMatrices}}
\end{figure}

It is possible to construct a configuration space from these fully labelled configurations. 

\begin{definition}[Fully labelled configuration spaces]
\label{def:robotGhost}
Consider a cell complex whose $0$-cells are indexed by the fully labelled configurations of $r$ distinguished robots and $R$ distinguished ghosts on a $K_{n,N}$ and whose $1$-cells correspond to a named robot switching places with a named ghost along an edge of the graph.  In terms of the associated permutation matrix, this type of movement corresponds to replacing a $2\times 2$ minor (with one entry in each of the four blocks).  If this minor is a $2\times 2$ identity matrix it is replaced with a matrix with ones on the anti-diagonal, and if it has ones on the anti-diagonal, it is replaced with an identity matrix.  Cubes in this complex correspond to a set of robot-ghost exchanges where the associated edges are pairwise disjoint.  We write $\config(r,R,n,N)$ to denote this cell complex.  

Another approach to describing the cubical structure of $\config(r,R,n,N)$ is via partitions.  Let $A$ be the set of vertices on the left of $K_{n,N}$, let $B$ be the set of vertices on the right, let $C$ be the set of robots and let $D$ be the set of ghosts.  A cubical cell of $\config(r,R,n,N)$ can be identified with a partition of the set $A \cup B \cup C \cup D$ into blocks of size $2$ and $4$.  Each block of size $2$ must contain one vertex from $A\cup B$ and one occupant from $C\cup D$ and these represent the robots and ghosts that are stationary throughout this cell.  Each block of size $4$ must contain exactly one element from each of $A$, $B$, $C$ and $D$, and this represents a particular robot and a particular ghost moving along a specific edge in $K_{n,N}$.  The dimension of the cube is equal to the number of blocks of size $4$.

\end{definition}

\begin{proposition}[Covering space]\label{prop:covering-space}
  The natural map from the configuration space $\config(r,R,n,N)$ to the configuration space $\config_r(n,N)$ is a covering map whose index is $r! R!$.
\end{proposition}

\begin{proof}
    The group $S_r \times S_R$ acts on $\config(r,R,n,N)$ by permuting the labels on the robots and ghosts.  Let $\sigma$ be a cubical cell in $\config(r,R,n,N)$ with  corresponding partition into blocks of size $2$ and $4$.  If we non-trivially permute the labels on the robots and/or the ghosts without changing the labels on the vertices, then the resulting partition is distinct from the original partition, and the resulting cube is distinct from the original cube.  In particular, the action on the complex is free. 
\end{proof}

\begin{remark}[Distinguishing robots and ghosts]
In the literature, it is common to consider the ($r!$)-fold cover of $\config_r(n,N)$ obtained by distinguishing the robots.  In the analogous situation for the classical braid group, this corresponds to the difference between the braid group and the pure braid group.   What we have described goes one step further by distinguishing the ghosts as well.  If we use a bar to indicate situations where robots and/or ghosts are not distinguishable then the standard configuration space described in \S\ref{sec:configurations} is:
\[
\config_r(n,N) = \config(\underline{r},\underline{R},n,N)
\]
The situation from the literature where robots (but not ghosts) are distinguished is $\config(r,\underline{R},n,N)$.  
\end{remark}

We note that the process of introducing ghosts, and distinguishing ghosts and robots, applies to  graph braid groups on any graph.

\begin{remark}[Transposing fully labelled configurations]
Given a fully labelled configuration of $r$ robots and $R$ ghosts on a $K_{n,N}$, one can switch the roles of the partitioned vertices and the occupying robots and ghosts, to yield a fully labelled configuration of $n$ robots and $N$ ghosts on a $K_{r,R}$.  This corresponds to transposing the associated permutation matrix, or  merely switching how robots/ghosts and left/right vertices are identified with the rows and columns of the matrix.
\end{remark}

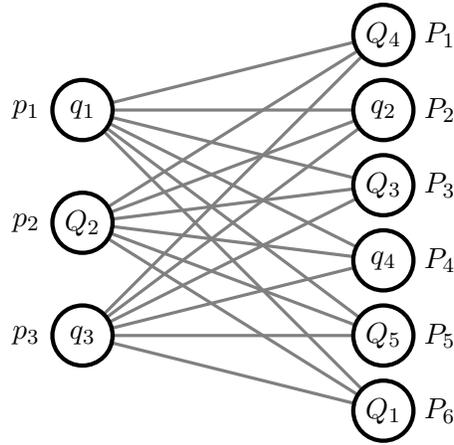
\begin{figure}[hptb]
\begin{tikzpicture} [scale=1, thick]

\begin{scope} [xshift=3cm]
\foreach \x in {1,2.5,4} \foreach \y in {0,1,2,3,4,5} 
\draw [very thick, gray] (-2,\x)--(2,\y);

\foreach \x in {1,2.5,4} \draw [ultra thick, fill=white] (-2,\x) circle [radius=0.4];
\foreach \x in {0,1,2,3,4,5} \draw [ultra thick, fill=white] (2,\x) circle [radius=0.4];

\node at (-2,4) {$q_1$}; \node at (-2.75,4) {$p_1$};
\node at (-2,2.5) {$Q_2$};\node at (-2.75,2.5) {$p_2$};
\node at (-2,1) {$q_3$};\node at (-2.75,1) {$p_3$};
\node at (2,5) {$Q_4$}; \node at (2.75,5) {$P_1$};
\node at (2,4) {$q_2$}; \node at (2.75,4) {$P_2$};
\node at (2,3) {$Q_3$};\node at (2.75,3) {$P_3$};
\node at (2,2) {$q_4$};\node at (2.75,2) {$P_4$};
\node at (2,1) {$Q_5$};\node at (2.75,1) {$P_5$};
\node at (2,0) {$Q_1$};\node at (2.75,0) {$P_6$};
\end{scope}

\end{tikzpicture}
\caption{The fully labelled configuration of four robots and five ghosts moving on a $K_{3,6}$ which is the transpose of the fully labelled configuration shown in Figure~\ref{fig:labelEverythingA}.}
\label{fig:labelEverythingB}
\end{figure}

\begin{example}[Transposing fully labelled configurations]
We return to the situation discussed in Example~\ref{ex:fullyLabelSymmetry}, with three distinguished robots and six distinguished ghosts on a $K_{4,5}$. At the level of the permutation matrix shown in Figure~\ref{fig:fourMatrices}, one simply transposes the $9\times 9$ matrix in the upper lefthand corner to produce a new matrix that describes $4$ distinguished robots and $5$ distinguished ghosts on a $K_{3,6}$ or, simply change how the matrix is being interpreted.  In Example~\ref{ex:fullyLabelSymmetry} we removed the subscripts on the rows to describe a $0$-cell in the configuration space $\config_3(4,5)$.  Here we remove the subscripts on the columns to produce a $9 \times 2$ matrix shown in the upper righthand corner which describes a $0$-cell in the configuration space $\config_4(3,6)$.   Returning to the graphical interpretation in Figure~\ref{fig:labelEverythingA}, the `transposed' version shown in Figure~\ref{fig:labelEverythingB} has the three $p$'s and six $P$'s now interpreted as left and right vertices of a $K_{3,6}$, the four $q$'s are the robots, and the five $Q$'s are the ghosts.  
\end{example}

Moving from one configuration to the other essentially corresponds to switching the roles of two $2$-colorings on $T = r+R = n+N$ vertices.  The first $2$-coloring describes the split between left and right and the second $2$-coloring describes the split between robots and ghosts.  Note that this type of switch is only possible because we are working with a complete bipartite graph.  We are unaware of any similar type of operation for other graphs.

\begin{proposition}[Common universal cover]\label{prop:commonUniversalCover}
Let $r, R, n,$ and $N$ be positive integers where $r+R = n+N$.  The eight configuration spaces: \[\config_r(n,N), \config_r(N,n), \config_R(n,N), \config_R(N,n),\]
\[\config_n(r,R), \config_n(R,r), \config_N(r,R),  \config_N(R,r)\] all share the same universal cover.
\end{proposition}

\begin{proof} 
First note that the four configuration spaces
\[
\config_r(n,N) \simeq \config_r(N,n) \simeq \config_R(n,N) \simeq \config_R(N,n)\]
are naturally identifiable via the swapping of left and right vertices and/or the swapping of robots and ghosts.  Similarly, the four configuration spaces
\[
\config_n(r,R) \simeq \config_n(R,r) \simeq \config_N(r,R) \simeq \config_N(R,r)\]
are also naturally identifiable.  And finally, by Proposition~\ref{prop:covering-space}, both families of identified configuration spaces have a common cover.  The cover of the first family is of index $r!R!$ and the cover of the second family is of index $n!N!$.  As a consequence, all eight have the same universal cover.
\end{proof}

\begin{example}[Common universal cover] \label{ex:common-universal-cover}
Consider the eight examples where $\{\{r,R\},\{n,N\}\} = \{\{3,6\},\{4,5\}\}$. The four configuration spaces with $\{r,R\} = \{3,6\}$ and $\{n,N\} = \{4,5\}$ all have exactly $\binom{9}{3} = \binom{9}{6} = 84$ vertices.  Similarly, the four configuration spaces with $\{r,R\} = \{4,5\}$ and $\{n,N\} = \{3,6\}$ have exactly $\binom{9}{4} = \binom{9}{5} = 126$ vertices.  The common cover of both of these spaces, in which everything is fully distinguished, would have $9! = 84 \cdot (3!6!) = 126 \cdot (4!5!) = 362,880$ vertices.
\end{example}

The following corollary allows us to assume that $r \leq n \leq N \leq R$ when computing the degree of connectivity at infinity.

\begin{corollary}[Ordering parameters]\label{cor:order}
Let $r,R,n,N$ be positive integers with $r+R=n+N$.  Among the eight configuration spaces with a common universal cover, there is at least one in which there are at least as many vertices on the right as on the left and all of the robots can fit on the lefthand side.  
\end{corollary}

\section{Vertex Links}\label{sec:vertex-links}

According to Theorem~\ref{thm:link-infinity} in the introduction, the connectivity at infinity of the universal cover of $\config_r(n,N)$ is determined by the minimal connectivity of certain links in $\config_r(n,N)$. In this section, we establish this minimal link connectivity in the special case where the parameters are ordered $r \leq n \leq N \leq R$ which, by Corollary~\ref{cor:order}, is not a serious restriction.  We begin by recalling the definition of a vertex link in a piecewise Euclidean complex built out of unit cubes and then acknowledge that this terminology breaks with our earlier conventions.

\begin{definition}[Vertex links]\label{def:vertex-links}
The link of a $0$-cell $v$ in a piecewise Euclidean cube complex is the metric sphere based at $v$ of  radius $\epsilon <1$.   For $0$-cells in $\config_r(n,N)$, the link admits a  natural simplicial structure where the $0$-simplices in $\Lk(v)$ correspond to $1$-cells of $\config_r(n,N)$ that contain $v$, and more generally, $k$-simplices in $\Lk(v)$ correspond to $(k+1)$-cubes that contain~$v$.  This simplicial complex is what is commonly known as a \emph{vertex link}.
\end{definition}

\begin{remark}[Vertex links]\label{rem:vertex-link}
Definition~\ref{def:vertex-links} is the one exception that we make to the convention described in Remark~\ref{rem:vertex-vs-0-cell}.  We use the standard terminology ``vertex link'' to denote the link of a $0$-cell in a configuration space, even though we continue to refrain from calling these $0$-cells ``vertices''.
\end{remark}

Vertex links in $\config_r(n,N)$ can also be described combinatorially.  

\begin{remark}[Combinatorial description of vertex links]\label{rem:comb-vertex-links}
A $0$-cell $v$ in $\config_r(n,N)$ corresponds to selecting $r$ vertices in a $K_{n,N}$ and designating them as being the vertices occupied by robots.  One can then define a poset whose elements are subsets $\sigma$ of the closed edges of $K_{n,N}$, where the edges in $\sigma$ are disjoint and each edge has one vertex that is occupied and the other is unoccupied. The order relation is given by containment. This poset is the face poset of $\Lk(v)$ and the geometric realization of the poset is the barycentric subdivision of $\Lk(v)$.  Because we are interested primarily in the topology of vertex links, we do not distinguish between the link of a vertex and its barycentric subdivision.
\end{remark}

Note that a $0$-cell of $\config_r(n,N)$ corresponds to a given arrangement of robots on $K_{n,N}$ and the link of this $0$-cell records the possible robot motions that start from this particular arrangement. Vertex links in $\config_r(n,N)$ also have a nice combinatorial representation in terms of \textit{chessboard complexes}.

\begin{definition}[Chessboard complexes]\label{def:chessboard-complexes}
To define a chessboard complex consider placements of rooks on an $m\times n$ chessboard so that no two rooks are in the same row or column. We call such placements \textit{permissible}. The collection of permissible placements can be viewed as a poset by considering one placement to be below another if the set of occupied squares in the first one are also occupied in the second.  The chessboard complex $\Delta_{m,n}$ is the simplicial realization of the poset of all permissible placements, on an $m\times n$ chessboard. 
\end{definition}

In \cite{MZ}, Meier and Zhang showed  that the vertex links in the configuration spaces of $r$ robots on the complete graph $K_{r+N}$ are homeomorphic to $\Delta_{r,N}$.  If a $0$-cell $v$ of $\config_r(n,N)$ corresponds to a configuration in which all robots are on the side of $K_{n,N}$ with $n$ vertices, then the link of this $0$-cell is the same as $\Delta_{r,N}$ as well. To see this, first note that the description of vertex links given in Remark~\ref{rem:comb-vertex-links} shows that the vertices occupied by ghosts on the same side as the robots, and the edges that emanate from these vertices, do not contribute anything to the link of $v$. Intuitively, this is because no robot is able to enter one of these edges or end at one of this vertices from this starting configuration of robots and ghosts.  So the link is the same as the link associated with a configuration of $r$ robots in a $K_{r,N}$ with all the robots on the left.  Similarly, in the complete graph $K_{r+N}$, the robots are unable to use the edges where both endpoints are occupied by robots, so here too would be the same as the link of $v$ in the complete bipartite graph $K_{r,N}$.

In general, when $v$ is a $0$-cell in $\config_r(n,N)$ corresponding to a configuration where there are robots on both sides, the robots on the left and the ghosts on the right determine one chessboard complex, and the robots on the right and the ghosts on the left determine a second chessboard complex.  More specifically, let $a$ be the number of robots on the $n$-side of $K_{n,N}$ and let $b$ be the number of robots on the $N$-side.  Then subsets of the $a$ robots may move to the $(N - b)$ vertices occupied by ghosts on the opposite side, giving a copy of $\Delta_{a, N-b}$.  Similarly subsets of the $b$ vertices may move to the $(n - a)$ vertices occupied by ghosts on their opposite side, yielding a copy of $\Delta_{b,n-a}$.  As these collections of motions are independent of each other, the resulting complex is the join of two chessboard complexes.

\begin{lemma}[Vertex links are joins]\label{lem:linkDesc}
Let $v$ be a $0$-cell in $\config_r(n,N)$ corresponding to a configuration where $a$ robots are on the $n$-side of $K_{n,N}$ and $b$  robots are on the $N$-side. Then the link of $v$ is 
\[ \Lk(v) \simeq \Delta_{a,N-b}\star \Delta_{b,n-a} \]
with the convention that the join of a complex $K$ with an empty set is just~$K$.
\end{lemma}

\begin{example}[Not a manifold]\label{ex:NotMfd}
Consider the case of two robots on a $K_{3,4}$.  If $v$ is a $0$-cell  corresponding to having one robot on each side of the graph, then the link of $v$ is the join of two points with three points, often called a theta-graph.  We note that in this case, as is typical, the link of a vertex is not a sphere and the configuration space is not a manifold.
\end{example}

Determining the connectivity of chessboard complexes proved to be rather difficult and the final result is stated below. See \cite{BjornerEtAl}, \cite{Ziegler}, \cite{Wachs}, \cite{Athan}, \cite{SharWachs}, and the many references cited in these papers for more detail.  

\begin{theorem}[Chessboard connectivity]\label{thm:ConnChess} 
The chessboard complex $\Delta_{m,n}$ is $(\nu_{m,n}-2)$-connected and not $(\nu_{m,n}-1)$-connected, where 
\[
\nu_{m,n}= \min\left\{m,n,\left\lf\frac{m+n+1}{3}\right\rf\right\}\,.
\]
\end{theorem}

When considering the join of two chessboard complexes, we can at least determine a lower bound on the connectivity.  The following lemma is stated explicitly as Corollary~3.12 in \cite{Jonsson} and follows from Exercise 4.1.17 in \cite{Hatcher}.

\begin{lemma}[Connectivity of joins]\label{lem:ConnJoin}
If $X$ is a $p$-connected simplicial complex and $Y$ is a $q$-connected simplicial complex, then $X \star Y$ is $(p+q+2)$-connected.
\end{lemma}

Theorem~\ref{thm:ConnChess} and Lemma~\ref{lem:ConnJoin} go a long way towards establishing much of our argument, particularly once we show that connectivity results about vertex links quickly lead to connectivity results about punctured vertex links, as is required by Theorem~\ref{thm:link-infinity}. A key fact about chessboard complexes is that they are \textit{vertex decomposible}, and in Section~4 of \cite{MZ} this fact is used to prove that connectivity is preserved when any closed simplex is removed (Theorem~4.7 of \cite{MZ}). 

\begin{theorem}[Punctured chessboard links]\label{thm:punctureChess}
Let $2 \leq m \leq n$ and let $\sigma$ be a closed simplex (possibly empty).  Then $\Delta_{m,n}$ with $\sigma$ removed is an  $(m-1)$-dimensional, vertex decomposable simplicial complex.  Thus $\Delta_{m,n}$ with $\sigma$ removed is $(\nu_{m,n}-2)$-connected, just like $\Delta_{m,n}$.
\end{theorem}

The same argument applies in the case of robots on complete bipartite graphs, because the links of vertices are either chessboard complexes or the joins of chessboard complexes, and vertex decomposibility is preserved by the join operation (Theorem~3.30 in \cite{Jonsson}).

\begin{corollary}[Punctured joins]\label{cor:puncturedJoin}
Let $\sigma$ be a closed simplex in the join $\J=\Delta_{m,n}\star \Delta_{p,q}$, including the possibility that $\sigma$ is the empty simplex.  Then $\J$ with $\sigma$ removed is $(\nu_{m,n} + \nu_{p,q}-2)$-connected.
\end{corollary}

\begin{proof}
Let $\sigma_0 = \sigma \cap \Delta_{m,n}$ and $\sigma_1 = \sigma \cap \Delta_{p,q}$.  Then 
\[
\J \setminus \sigma
= \left(\Delta_{m,n} \setminus \sigma_0\right)
\star \left(\Delta_{p,q} \setminus \sigma_1\right)\,.
\]
By Theorem~\ref{thm:punctureChess}, $\J \setminus \sigma$ is the join of a $(\nu_{m,n} - 2)$-connected complex and a $(\nu_{p,q}-2)$-connected complex.  The formula follows by Lemma~\ref{lem:ConnJoin}.
\end{proof}

\begin{table}[hptb]
\caption{We compute lower bounds on the connectivity of the links (and punctured links) of vertices in $\config_3(4,5)$ using Theorem \ref{thm:ConnChess}, Lemmas \ref{lem:linkDesc} and \ref{lem:ConnJoin}, and Corollary~\ref{cor:puncturedJoin}.}
\label{table:k45}
\setlength{\tabcolsep}{1.5mm} 
\def\arraystretch{1.25} 
\centering
\begin{tabular}{|c|c|c|}
  \hline
  \begin{tabular}{@{}c@{}}Robot configuration\\ corresponding to a \\vertex $v$ in $\config_3(4,5)$ \end{tabular}  &   $\Lk(v)$    &   Connectivity of $\Lk(v)$
  \\ \hline
  \begin{tikzpicture}[scale=0.7,baseline={(2,1.5)}]
\foreach \x in {0,4/3,8/3,4} \foreach \y in {0,1,2,3,4} 
\draw [gray] (-2,\x)--(2,\y);

\foreach \x in {0,4/3,8/3,4} \draw [ultra thick, fill=white] (-2,\x) circle [radius=0.2];
\foreach \x in {0,1,2,3,4} \draw [ultra thick, fill=white] (2,\x) circle [radius=0.2];

\draw [ultra thick, fill=black] (-2,0) circle [radius=0.2];
\draw [ultra thick, fill=black] (-2,4/3) circle [radius=0.2];
\draw [ultra thick, fill=black] (-2,8/3) circle [radius=0.2];
\addvmargin{1mm}\end{tikzpicture}   &    $\Delta_{3,5}\star\Delta_{0,1}=\Delta_{3,5}$   & $\min\{3,5,\lf\frac{9}{3}\rf\}-2=1$
  \\ \hline
  \begin{tikzpicture}[scale=0.7, baseline={(2,1.5)}]
\foreach \x in {0,4/3,8/3,4} \foreach \y in {0,1,2,3,4} 
\draw [gray] (-2,\x)--(2,\y);

\foreach \x in {0,4/3,8/3,4} \draw [ultra thick, fill=white] (-2,\x) circle [radius=0.2];
\foreach \x in {0,1,2,3,4} \draw [ultra thick, fill=white] (2,\x) circle [radius=0.2];

\draw [ultra thick, fill=black] (2,2) circle [radius=0.2];
\draw [ultra thick, fill=black] (-2,4/3) circle [radius=0.2];
\draw [ultra thick, fill=black] (-2,8/3) circle [radius=0.2];
\addvmargin{1mm}\end{tikzpicture}   &    $\Delta_{2,4}\star\Delta_{1,2}$   & \begin{tabular}{@{}c@{}}$(\min\{2,4,\lf\frac{7}{3}\rf\}-2)+$ \\ $(\min\{1,2,\lf\frac{4}{3}\rf\}-2)+2=1$\end{tabular}
  \\ \hline
  \begin{tikzpicture}[scale=0.7,baseline={(2,1.5)}]
\foreach \x in {0,4/3,8/3,4} \foreach \y in {0,1,2,3,4} 
\draw [gray] (-2,\x)--(2,\y);

\foreach \x in {0,4/3,8/3,4} \draw [ultra thick, fill=white] (-2,\x) circle [radius=0.2];
\foreach \x in {0,1,2,3,4} \draw [ultra thick, fill=white] (2,\x) circle [radius=0.2];

\draw [ultra thick, fill=black] (2,2) circle [radius=0.2];
\draw [ultra thick, fill=black] (-2,4/3) circle [radius=0.2];
\draw [ultra thick, fill=black] (2,4) circle [radius=0.2];
\addvmargin{1mm}\end{tikzpicture}   & $\Delta_{1,3}\star\Delta_{2,3}$  & \begin{tabular}{@{}c@{}}$(\min\{1,3,\lf\frac{5}{3}\rf\}-2)+$ \\ $(\min\{2,3,\lf\frac{6}{3}\rf\}-2)+2=1$\end{tabular}
  \\ \hline
    \begin{tikzpicture}[scale=0.7,baseline={(2,1.5)}]
\foreach \x in {0,4/3,8/3,4} \foreach \y in {0,1,2,3,4} 
\draw [gray] (-2,\x)--(2,\y);

\foreach \x in {0,4/3,8/3,4} \draw [ultra thick, fill=white] (-2,\x) circle [radius=0.2];
\foreach \x in {0,1,2,3,4} \draw [ultra thick, fill=white] (2,\x) circle [radius=0.2];

\draw [ultra thick, fill=black] (2,2) circle [radius=0.2];
\draw [ultra thick, fill=black] (2,1) circle [radius=0.2];
\draw [ultra thick, fill=black] (2,4) circle [radius=0.2];
\addvmargin{1mm}\end{tikzpicture}   & $\Delta_{0,2}\star\Delta_{3,4}=\Delta_{3,4}$  & $\min\{3,4,\lf\frac{8}{3}\rf\}-2=0$
  \\ \hline
  \end{tabular}
\end{table}

\begin{remark}[Exact bounds]\label{rem:ExactBounds}
We apply these results to $\config_3(4,5)$ and determine the connectivity of vertex links in Table \ref{table:k45}.  The reader should note, however, that at the moment, we have only proved that these values are lower bounds on the connectivities of these joins.  Some of these joins might conceivably be more highly connected.  We address the issue of exactness of the minimum value in the proof of Theorem~\ref{thm:ExactBounds}. In general, the lower bound given in Lemma~\ref{lem:ConnJoin} is usually sufficient for determining the minimal non-connectivity of links but there is one case where we need to analyze a non-trivial join (Lemma~\ref{lem:exception}).
\end{remark}

\section{Computations}\label{sec:computations}

Building from the results in the previous section, we are now able to complete the computation of the minimum connectivity of all the vertex links and the minimum non-connectivity of some vertex link, minimized over the $0$-cells of the configuration space $\config_r(n,N)$. The goal of this section is to prove the following result.

\begin{theorem}[Link connectivity]\label{thm:ExactBounds}
Let $2 \leq r \leq n \leq N\le R$ be integers such that $r+R=n+N$ and let 
\[\ell = \min\left\{r, \thirdfloor{r+n+1}, \thirdfloor{n+N} \right\}.\]
Then the link of every $0$-cell in $\config_r(n,N)$ is at least $(\ell - 2)$-connected. Moreover, there exists a $0$-cell whose link is not $(\ell -1)$-connected.  
\end{theorem}

We first determine the minimum connectivity of the vertex links only among those $0$-cells of a particular type $(a,b,c,d)$ (Definition~\ref{def:config-type}).  Recall that this means $a+b=r$,  $a+c=n$, $b+d=N$ and $c+d=R$. And note that for fixed values of $r$, $n$, $N$ and $R$, these equations allow one to solve for the other three once one of $a$, $b$, $c$ or $d$ is known. Since the link of a $0$-cell of type $(a,b,c,d)$ is isomorphic to the join of chessboard complexes $\Delta_{a,d} \star \Delta_{b,c}$, the connectivity of the link of a $0$-cell of type $(a,b,c,d)$ is at least $(\nu_{a,d} - 2) + (\nu_{b,c} - 2) + 2$, which is equal to

\[ 
\min\left\{a,d, \thirdfloor{a+d+1} \right\} + 
\min\left\{b,c, \thirdfloor{b+c+1} \right\}-2.
\]

So, in order to find the minimum connectivity, we need to compute
\begin{equation}\label{eq:min-min}
\min_{(a,b,c,d)}\left\{
\min\left\{a, d, \thirdfloor{a + d +1} \right\} + 
\min\left\{b, c, \thirdfloor{b + c +1} \right\} 
\right\}.
\end{equation}
where the outside minimum is taken over all possible $(a,b,c,d)$ types which are $(r,R,n,N)$ solutions.

We analyze this minimum of a sum of minimums in Equation~\ref{eq:min-min} by interchanging the order in which the minimums are taken.  We first
look at the values that can occur for all nine possible summations arising from the formula above (a choice of $a$, $d$ or $\thirdfloor{a+d+1}$ in the first summand followed by a choice of $b$, $c$ or $\thirdfloor{b+c+1}$ in the second summand), and then determine the minimum of each one over all possible $(a,b,c,d)$.

\begin{figure}[hbpt]
\begin{tabular}{c|ccc}
& $b$ & $c$ & $\thirdfloor{b+c+1}$\\ [5pt]
\hline 
$a$ & \textcolor{red}{$a+b$} & \textcolor{red}{$a+c$} & \textcolor{blue}{$a+\thirdfloor{b+c+1}$} \\ [5pt]
$d$ & \textcolor{red}{$d+b$} & \textcolor{red}{$d+c$} & \textcolor{blue}{$d+\thirdfloor{b+c+1}$}\\ [5pt]
$\thirdfloor{a+d+1}$ & \textcolor{blue}{$\thirdfloor{a+d+1}+b$}
& \textcolor{blue}{$\thirdfloor{a+d+1}+c$} 
& $\thirdfloor{a+d+1} + \thirdfloor{b+c+1}$
\end{tabular}
\caption{The nine possible summations.\label{fig:nine-possibilites}\label{fig:color-coded-table}}
\end{figure}

\begin{remark}[Nine cases]
The nine possible sums are displayed in Figure~\ref{fig:color-coded-table}.  Since $a$ is the number of robots on the side of $K_{n,N}$ with $n$ vertices, $r$ is the total number of robots, and we are assuming $r\leq n\leq N\leq R$, there are $r+1$ possibilities for $a$, and hence for $(a,b,c,d)$. In particular there are $r+1$ different sums that need to be checked for each entry in Figure~\ref{fig:color-coded-table}.  
\end{remark}

The nine possibilities are grouped according to the number of floor functions they contain and we analyze each group one at a time.  We proceed by first finding the minimum value of the $4(r+1)$ values corresponding to the entries of Figure~\ref{fig:color-coded-table} with no floor functions. 

\begin{claim}[No floor minimum]\label{claim:no-floors}
The minimum of the (red) entries with no floor function,  over all possible values of $(a,b,c,d)$, is $r$.
\end{claim}

\begin{proof}
Since $(a,b,c,d)$ is an $(r,n,N,R)$-solution, the entries in the above table that do not involve the floor function are constant across all possible values of $(a,b,c,d)$. Regardless of the configuration, $a+b = r$, $a+c=n$, $b+d=N$ and $c+d=R$. Since we are assuming $r \leq n \leq N \leq R$, the minimum of these four possible values is $r$.
\end{proof}

Next we find the minimum value of the $4(r+1)$ values corresponding to the entries with a single floor function. 

\begin{claim}[One floor minimum]\label{claim:1floor}
The minimum of the (blue) entries with exactly one floor function, over all possible values of $(a,b,c,d)$, is $\thirdfloor{r+n+1}$.
\end{claim}

\begin{proof}
Each of the four one floor summations simplify to a single floor function involving $r$, $R$, $n$, or $N$ with the all summations inside the floor function. For example,
\begin{align*}
\thirdfloor{b + c + 1} + a &= \thirdfloor{3a + b + c + 1}\\
&= \thirdfloor{(a + b) + (a + c) + a + 1}\\
&= \thirdfloor{r + n + a + 1}.
\end{align*}
Similarly, the other three one floor summations  simplify to 
\[
\thirdfloor{R + N + d + 1},
\thirdfloor{R + n + c + 1}, \text{ and }
\thirdfloor{r + N + b + 1}.
\]
By our assumption that $r \leq n \leq N \leq R$ we know that $r + n$ is minimal among $r + n$, $r + N$, $R + n$, and  $R + N$. Moreover, the minimum of these four terms will occur when $a=0$.  So among these four terms the minimum will be
\[
\thirdfloor{r + n + a + 1} = 
\thirdfloor{r + n + 0 + 1} = 
\thirdfloor{r + n+1 }.
\]
\end{proof}

Finally, we determine the minimum of the $r+1$ values arising from  the entry with two floor functions. 

\begin{claim}[Two floor minimum]\label{claim:2floors}
The minimum of the  (black) entries of the form $\thirdfloor{a+d+1} + \thirdfloor{b+c+1}$, over all possible values of $(a,b,c,d)$, is $\thirdfloor{n+N}$.
\end{claim}

The statement and proof of this result is more delicate than the other two. In fact, Claim~\ref{claim:2floors} is only true because of our standing assumption that $\min\{r,R,n,N\}\geq 2$. Before giving the proof of this claim, we first show how to rewrite the sum of two floor functions as a single floor function.

\begin{lemma}[Adding floors]\label{lem:add-floors}
Let $p$ and $q$ be integers. If $i \equiv q \mod 3$ is the remainder of $q$ mod $3$ with $i \in \{0,1,2\}$, then
\[
\thirdfloor{p} + \thirdfloor{q} 
= \thirdfloor{p + q - i}.
\]
\end{lemma}

\begin{proof}
For any integer $k$, 
\[
\thirdfloor{p} + \thirdfloor{q} = \left(\thirdfloor{p}+k\right) + \left(\thirdfloor{q}-k\right) = \thirdfloor{p+3k} + \thirdfloor{q-3k}. 
\]
In particular, when $k$ is chosen so that $q = 3k + i$ with $i \in \{0, 1,2\}$, 
\[ 
\thirdfloor{p} + \thirdfloor{q} = \thirdfloor{p+q-i} + \thirdfloor{i} = \thirdfloor{p+q-i}.
\]
\end{proof}

Since the single floor function of Lemma~\ref{lem:add-floors} depends on the value of $q \mod 3$, this lemma enables us to compute the value of the sum listed in Claim~\ref{claim:2floors}, but at the cost of splitting the answer into three cases.

\begin{corollary}[Three cases]\label{cor:threecases}
If $(a,b,c,d)$ is an $(r,n,N,R)$-solution, then 
\[
\thirdfloor{a+d+1} + \thirdfloor{b+c+1} =  
\begin{cases} 
\thirdfloor{n+N+2} &\text{ if } b+c \equiv 2 \mod 3\\ 
\thirdfloor{n+N+1} &\text{ if } b+c \equiv 0 \mod 3\\ 
\thirdfloor{n+N} &\text{ if } b+c \equiv 1 \mod 3 
\end{cases}
\]
In particular, if there is a configuration with $b + c \equiv 1 \mod 3$, then the minimum of the sum of the two floor functions over all configurations is $\thirdfloor{n+N}$.  
\end{corollary}

\begin{proof}
Let $i$ be the remainder of $b+c+1$ mod $3$, with $i \in \{0,1,2\}$.  By Lemma~\ref{lem:add-floors}, $\thirdfloor{a+d+1} + \thirdfloor{b+c+1} = \thirdfloor{n+N+2-i}$, which simplifies as shown for the three possible values of $i$. Since $\thirdfloor{n+N}$ is the smallest of the three values it is the minimum if it occurs. 
\end{proof}

From this Corollary it is easy to complete the proof of Claim~\ref{claim:2floors}.

\begin{proof}[Proof of Claim~\ref{claim:2floors}]
Moving a robot from one side of a $K_{n,N}$ to the other changes the sum $b+c$ by $2$.  Since $r, n\geq 2$ by our standing assumption, there are at least three configurations with the associated sum $b+c$ yielding three different remainders modulo $3$.  Thus  Corollary~\ref{cor:threecases} implies Claim~\ref{claim:2floors}.
\end{proof}

Together, Claims~\ref{claim:no-floors}, \ref{claim:1floor} and~\ref{claim:2floors} establish the following lower bound.

\begin{proposition}[Lower Bound]\label{prop:lower-bound}
Let $2 \leq r \leq n \leq N\le R$ be integers such that $r+R=n+N$ and let 
\[\ell = \min\left\{r, \thirdfloor{r+n+1}, \thirdfloor{n+N} \right\}.\]
The link of every vertex in $\config_r(n,N)$ is at least $(\ell - 2)$-connected.
\end{proposition}

This connectivity bound is actually sharp as we now show. In particular, by Theorem~\ref{thm:ConnChess} we know the exact connectivity when the vertex link itself is a single chessboard complex, i.e. when it corresponds to a configuration with all of the robots on one side.  
It turns out that such configurations almost always yield the minimum connectivity.  In the proof of the one floor minimum, Claim~\ref{claim:1floor}, the minimum is realized when $a=0$, which means that all of the robots are on one side of the graph, and we know the exact connectivity of this vertex link by Theorem~\ref{thm:ConnChess}.   
In the proof of the no floor minimum, Claim~\ref{claim:no-floors}, the minimum of $r$ is realized regardless of the configuration, so we can elect to use a configuration where all of the robots are on one side, so that the vertex link is a single chessboard complex and we know its exact connectivity. 

The third possible value $\thirdfloor{n+N}$, however, might only be produced in situations where there are robots on both sides, making the vertex link a join of two chessboard complexes, and in this case all we have is a lower bound on the connectivity (Remark~\ref{rem:ExactBounds}).  Luckily, the types of situations where the third entry is strictly necessary are very rare and their exact connectivity can be handled explicitly.

\begin{lemma}[Exceptional case]\label{lem:exception}
If the minimum value for the connectivity of the link of a $0$-cell in $\config_r(n,N)$ is determined by a non-trivial join of two chessboard complexes then $2\leq r=n=N=R$ and $r\equiv 1\mod{3}$.
\end{lemma}

\begin{proof}
The only situation where we need to consider the connectivity of a non-trivial join of two chessboard complexes is when  $\thirdfloor{n+N}$ is strictly less than both $r$ and $\thirdfloor{r+n+1}$.  If $\thirdfloor{n+N}$ is strictly less than $\thirdfloor{r+n+1}$, then  $n+N < r+n+1$.  This inequality simplifies to $N \leq r$ which forces $r=n=N$ since Theorem~\ref{thm:ExactBounds} assumes $r \leq n \leq N \leq R$. Since $r+R=n+N$ we also have $R=n$, and thus all four values are equal. Eliminating $n$ and $N$, the three values being minimized simplify to $r$, $\thirdfloor{2r+1}$ and $\thirdfloor{2r}$ and it is clear that $\thirdfloor{2r} < r$ for all positive integers $r$. If $r\equiv 0$ or $2\mod{3}$, then $\thirdfloor{2r}=\thirdfloor{2r+1}$, the third entry is not the unique minimum value, and the minimum connectivity is realized by a link corresponding to a configuration where all of the robots are on one side.  If, however, $r\equiv 1\mod{3}$ then $2r+1\equiv 0\mod{3}$ while $2r\equiv 2\mod{3}$ so that $\thirdfloor{2r}<\thirdfloor{2r+1}$ as desired.
\end{proof}


We can now complete the proof of the main result of this section.

\begin{proof}[Proof of Theorem~\ref{thm:ExactBounds}]
Proposition \ref{prop:lower-bound} shows that the link of every vertex in $\config_r(n,N)$ is at least $(\ell - 2)$-connected. Thus, it remains to show there is a vertex whose link is not $(\ell-1)$-connected. When $\ell$ is either $r$ or $\thirdfloor{r+n+1}$, this minimum can be realized by a configuration with all of the robots on one side.  This means that the link of this vertex is a chessboard complex, not a join of chessboard complexes, and by Theorem \ref{thm:ConnChess} this link is not $(\ell-1)$-connected. 

If $\ell$ is solely achieved by the expression $\thirdfloor{n+N}$, then by Lemma~\ref{lem:exception}, $r=n=N=R$ and $r \equiv 1 $ mod 3, which implies that $a=d$ and $b=c$. If we choose $b=c=2$ and $a=d=r-2$, then $b+c \equiv 1$ mod 3 as required by Corollary~\ref{cor:threecases}. In this special case, the link of the associated vertex is $\Delta_{r-2,r-2}\star \Delta_{2,2}$. Since $\Delta_{2,2}$ is just two points, the full link is the suspension of $\Delta_{r-2,r-2}$, which means the connectivity has increased by exactly one.
\end{proof}

Given any positive integers $r$, $R$, $n$ and $N$ where $r+R = n+N$, the universal cover of the associated configuration space $\config_r(n,N)$ is isomorphic to the universal cover of a  configuration space where these variables are permuted so that they are in increasing order (see Corollary~\ref{cor:order}).  Thus we can remove the assumption that the parameters are ordered. Further, we can apply Theorem~\ref{thm:link-infinity} to convert the connectivity of vertex links to the connectivity at infinity of the universal cover.  These modifications transform Theorem~\ref{thm:ExactBounds} into our Main Theorem.

\begin{maintheorem}
Consider $r$ robots on a complete bipartite graph $K_{n,N}$. Let $R$ be the number of open vertices on $K_{n,N}$ so that $r+R=n+N$ is the total number of vertices.  Moreover, to avoid trivial cases, assume that $r$, $n$, $N$ and $R$ are all at least $2$.  Let $\ell_0 = \min\{r,R,n,N\}$, $\ell_1=\min\{r,R\}+\min\{n,N\}+1$, and $\ell_2=r+R=n+N$.  If 
\[\ell = \min\left\{
\ell_0, \thirdfloor{\ell_1}, \thirdfloor{\ell_2}
\right\},\]
then the universal cover of the combinatorial configuration space of $r$ robots on $K_{n,N}$ is $(\ell - 2)$-connected at infinity but not $(\ell - 1)$-connected at infinity.  
\end{maintheorem}

Although the space $\config_r(n,N)$ is rarely a manifold (Example~\ref{ex:NotMfd}), it is much more common for it to be a finite duality complex, which is a generalization of a compact manifold.  
Recall that when $X$ is a finite, aspherical cell complex with dimension $d$, and $\widetilde{X}$ is $(d-2)$-connected at infinity, then $X$ is said to be $d$-dimensional duality complex (see  \cite{geog}).  In our context, we are able to give a precise description of the conditions necessary for $\config_r(n,N)$ to be a duality complex.

\begin{corollary}[Duality complex]
Assume $2 \leq r \leq n \leq N$.  The configuration space $\config_r(n,N)$ is an $r$-dimensional duality complex if and only if $n \geq 2r-1$.
\end{corollary}

\begin{proof}
Given our ordered parameters, the dimension of $\config_r(n,N)$ is $r$.  So we need to show that the universal cover of $\config_r(n,N)$ is $(r-2)$-connected at infinity.  Thus by Theorem~\ref{thm:ExactBounds} we need 
\[\ell = \min\left\{r, \thirdfloor{r + n + 1}, \thirdfloor{n+N}\right\} = r.\]
In particular we need
\[ \thirdfloor{r + n + 1} \geq r, \] or equivalently \[ \thirdfloor{r + n + 1 - 3r} \geq 0.\]
Therefore it is necessary that $n \geq 2r - 1$.  

Since we are assuming $N \geq n$, if $n \geq 2r-1$ then $N  \geq 2r-1$, and so $n + N \geq 4r - 2$.  But $4r - 2 \geq 3r$ when $r \geq 2$, so
\[ 
\thirdfloor{n+ N} \geq \thirdfloor{4r - 2} \geq \thirdfloor{3r} = r, 
\]
which shows that $n+1 \geq 2r$ is also sufficient.
\end{proof}

\begin{remark}[Multipartite graphs] 
The reader should note that our argument relies heavily on the symmetry described in Corollary~\ref{cor:order} in order to simplify our analysis. This symmetry, however, does not immediately extend to complete multipartite graphs. It is also the case that links of $0$-cells in the configuration spaces for robots on a multipartite graph are more complicated than what occurs in the complete and complete bipartite cases.  Hence an analysis of the multipartite graph case will require  different ideas.
\end{remark}

\bibliographystyle{alpha}
\bibliography{references}

\begin{thebibliography}{BLV{\v{Z}}94}

\bibitem[Abr02]{birthday}
Aaron Abrams.
\newblock Configuration spaces of colored graphs.
\newblock {\em Geom. Dedicata}, 92:185--194, 2002.
\newblock Dedicated to John Stallings on the occasion of his 65th birthday.

\bibitem[AG02]{AbramsGhrist}
Aaron Abrams and Robert Ghrist.
\newblock Finding topology in a factory: configuration spaces.
\newblock {\em Amer. Math. Monthly}, 109(2):140--150, 2002.

\bibitem[Ath04]{Athan}
Christos~A. Athanasiadis.
\newblock Decompositions and connectivity of matching and chessboard complexes.
\newblock {\em Discrete Comput. Geom.}, 31(3):395--403, 2004.

\bibitem[BH99]{BridsonHaefliger}
Martin~R. Bridson and Andr\'e Haefliger.
\newblock {\em Metric spaces of non-positive curvature}, volume 319 of {\em
  Grundlehren der Mathematischen Wissenschaften [Fundamental Principles of
  Mathematical Sciences]}.
\newblock Springer-Verlag, Berlin, 1999.

\bibitem[BLV{\v{Z}}94]{BjornerEtAl}
A.~Bj{\"o}rner, L.~Lov{\'a}sz, S.~T. Vre{\'c}ica, and R.~T.
  {\v{Z}}ivaljevi{\'c}.
\newblock Chessboard complexes and matching complexes.
\newblock {\em J. London Math. Soc. (2)}, 49(1):25--39, 1994.

\bibitem[BM01]{BradyMeier}
Noel Brady and John Meier.
\newblock Connectivity at infinity for right angled {A}rtin groups.
\newblock {\em Trans. Amer. Math. Soc.}, 353(1):117--132, 2001.

\bibitem[BMM03]{converse}
Noel Brady, Jon McCammond, and John Meier.
\newblock Local-to-asymptotic topology for cocompact {$\rm CAT(0)$} complexes.
\newblock {\em Topology Appl.}, 131(2):177--188, 2003.

\bibitem[Bro94]{brown}
Kenneth~S. Brown.
\newblock {\em Cohomology of groups}, volume~87 of {\em Graduate Texts in
  Mathematics}.
\newblock Springer-Verlag, New York, 1994.
\newblock Corrected reprint of the 1982 original.

\bibitem[Far06]{FarleyHomology}
Daniel Farley.
\newblock Homology of tree braid groups.
\newblock In {\em Topological and asymptotic aspects of group theory}, volume
  394 of {\em Contemp. Math.}, pages 101--112. Amer. Math. Soc., Providence,
  RI, 2006.

\bibitem[Far07]{FarleyCohoTree}
Daniel Farley.
\newblock Presentations for the cohomology rings of tree braid groups.
\newblock In {\em Topology and robotics}, volume 438 of {\em Contemp. Math.},
  pages 145--172. Amer. Math. Soc., Providence, RI, 2007.

\bibitem[FS05]{FarleySab}
Daniel Farley and Lucas Sabalka.
\newblock Discrete {M}orse theory and graph braid groups.
\newblock {\em Algebr. Geom. Topol.}, 5:1075--1109, 2005.

\bibitem[Geo08]{geog}
Ross Geoghegan.
\newblock {\em Topological methods in group theory}, volume 243 of {\em
  Graduate Texts in Mathematics}.
\newblock Springer, New York, 2008.

\bibitem[Hat02]{Hatcher}
Allen Hatcher.
\newblock {\em Algebraic topology}.
\newblock Cambridge University Press, Cambridge, 2002.

\bibitem[Jon08]{Jonsson}
Jakob Jonsson.
\newblock {\em Simplicial complexes of graphs}, volume 1928 of {\em Lecture
  Notes in Mathematics}.
\newblock Springer-Verlag, Berlin, 2008.

\bibitem[KP12]{KoPark}
Ki~Hyoung Ko and Hyo~Won Park.
\newblock Characteristics of graph braid groups.
\newblock {\em Discrete Comput. Geom.}, 48(4):915--963, 2012.

\bibitem[MZ13]{MZ}
John Meier and Liang Zhang.
\newblock Connectivity at infinity for braid groups on complete graphs.
\newblock {\em Homology Homotopy Appl.}, 15(1):303--311, 2013.

\bibitem[Ram18]{Ramos}
Eric Ramos.
\newblock Stability phenomena in the homology of tree braid groups.
\newblock {\em Algebr. Geom. Topol.}, 18(4):2305--2337, 2018.

\bibitem[Sch18]{Scheirer}
Steven Scheirer.
\newblock Topological complexity of {$n$} points on a tree.
\newblock {\em Algebr. Geom. Topol.}, 18(2):839--876, 2018.

\bibitem[SW07]{SharWachs}
John Shareshian and Michelle~L. Wachs.
\newblock Torsion in the matching complex and chessboard complex.
\newblock {\em Adv. Math.}, 212(2):525--570, 2007.

\bibitem[Wac03]{Wachs}
Michelle~L. Wachs.
\newblock Topology of matching, chessboard, and general bounded degree graph
  complexes.
\newblock {\em Algebra Universalis}, 49(4):345--385, 2003.
\newblock Dedicated to the memory of Gian-Carlo Rota.

\bibitem[Zie94]{Ziegler}
G{\"u}nter~M. Ziegler.
\newblock Shellability of chessboard complexes.
\newblock {\em Israel J. Math.}, 87(1-3):97--110, 1994.

\end{thebibliography}

\end{document}